\documentclass[12pt,twoside,reqno]{amsart}

\usepackage{amsfonts,amssymb,amsmath,mathrsfs}
\usepackage{lineno,hyperref}
\usepackage{comment}
\usepackage{mathtools}
\usepackage{color}
\usepackage{xcolor}
\usepackage{amssymb}
\usepackage{cite}
\usepackage{calc}
\usepackage{mathabx}

\usepackage[paperheight=10.50in,paperwidth=8.25in,margin=1.5in,heightrounded]{geometry}

\newcommand*{\caL}{\ensuremath{\mathcal{L}}}		
\newcommand*{\caH}{\ensuremath{\mathcal{H}}}		
\newcommand*{\caG}{\ensuremath{\mathcal{G}}}		

\newcommand{\supp}{\mathrm{supp}}
\newcommand{\scc}{\mathrm{scc}}
\newcommand{\sccl}{\mathrm{sc}}

\newcommand{\gh}{\mathcal{G}^{\sigma}\mathcal{H}_{r,\rho}}

\newcommand{\R}{\mathbb{R}}
\newcommand{\N}{\mathbb{N}}
\newcommand{\Z}{\mathbb{Z}}
\newcommand{\T}{\mathbb{T}}
\newcommand{\cinf}{\mathscr{C}^{\infty}}

\newcommand{\Fc}{\mathscr{F}}
\newcommand{\C}{\mathbb{C}}
\newcommand{\Op}{\mathrm{Op}}

\def\Op{ {\operatorname{Op}} }

\newcommand{\condA}{$\mathrm{(}\mathscr{A}\mathrm{)}$}

\newtheorem{thm}{Theorem}[section]
\newtheorem{defn}[thm]{Definition}
\newtheorem{prop}[thm]{Proposition}
\newtheorem{lemma}[thm]{Lemma}
\newtheorem{cor}[thm]{Corollary}
\newtheorem{rem}[thm]{Remark}

\title[{\tiny{Time-periodic operators on asymptotically Euclidean manifolds}}]
{Global hypoellipticity\\for a class of time-periodic operators\\on asymptotically Euclidean manifolds}

\author[F. Avila Silva]{Fernando de \'Avila Silva}
\address{Department of Mathematics  \\ Federal University of Paran\'a \\ Curitiba, CEP 81531-980, Caixa Postal 19081, Paran\'a, Brazil}
\email{fernando.avila@ufpr.br}

\author[M. Bonino]{Matteo Bonino}
\address{Dipartimento di Matematica ``G. Peano'' \\Universit\`a di Torino\\
	Via Carlo Alberto 10\\
	10123 Torino\\
	Italy}
\email{matteo.bonino@unito.it}

\author[S. Coriasco]{Sandro Coriasco}
\address{Dipartimento di Matematica ``G. Peano'' \\Universit\`a di Torino\\
	Via Carlo Alberto 10\\
	10123 Torino\\
	Italy}
\email{sandro.coriasco@unito.it}

\begin{document}
	
	\begin{abstract}
		We introduce time-periodic Gevrey-Sobolev-Kato spaces on asymptotically Euclidean manifolds
		and study their characterisation throughout Fourier expansions associated with suitable elliptic operators. 
		As an application, we study the global hypoellipticity problem for a naturally associated 
		class of time-periodic evolution equations.
	\end{abstract}


	\maketitle

	\tableofcontents

\section{Introduction}\label{sec:intro}

Let $X$ be a $d$-dimensional asymptotically Euclidean manifold, that is, a smooth compact manifold $X$ with boundary, 
equipped with a Riemannian metric $g$ satisfying suitable hypotheses close to the boundary $\partial X$. We denote by $\rho_X$
a fixed \textit{boundary defining function} for $X$ (see \cite{ME} and Appendix \ref{app:scattering} for details). 
In this paper we introduce a family of Gevrey-Sobolev-Kato spaces on $\T^n\times X$, where $\T = \R /(2\pi \Z)$ is the $1$-dimensional torus, 
and study the global hypoellipticity properties of a naturally associated class of time-periodic evolution operators.

In short, the spaces under study are constructed as follows (see Section \ref{sec:spaces} below for details). 
Let $\dot{\mathscr{C}}^\infty(X)=\bigcap_{k\in\Z_+}\rho_X^k C^\infty(X)$ be the space of 
smooth functions on $X$ vanishing of infinite order at $\partial X$, and let $(\dot{\mathscr{C}}^\infty(X))'$ be its dual. For  $r, \rho \in \mathbb{R}$ consider 
$H^{r,\rho}(X)$, a Sobolev-Kato (or \textit{weighted Sobolev}) space on $X$ (see, e.g., \cite[p. 76]{ME}).
Then, for any fixed $\sigma > 1$  and $C > 0$, we study the spaces of all functions $u \in \mathscr{C}^{\infty}(\mathbb{T}^n; (\dot{\mathscr{C}}^\infty(X))')$ such that
\begin{equation}\label{eq:GHnorm_Intro}
	\sup_{\gamma \in \mathbb{Z}_+^n} \left\{ C^{-|\gamma|} (\gamma!)^{-\sigma} \sup_{t \in \mathbb{T}^n} \| \partial_t^{\gamma} u(t) \|_{H^{r,\rho}(X)} \right\} < \infty.
\end{equation}
We then consider suitable projective and inductive limits of such space families and of their duals.
 
The key tool in our analysis is the characterisation of these spaces in terms of a 
discretisation process, obtained by means of Fourier expansions generated by a suitable elliptic operator $P$ on $X$. 
Specifically, if $\{ \phi_j \}_{j \in \mathbb{N}^*} \subset \dot{\mathscr{C}}^\infty(X)$  is an associated orthonormal basis of eigenfunctions of $P$, 
corresponding to the eigenvalues $\{ \lambda_j \}_{j \in \mathbb{N}^*}$, we  expand elements  
$u \in \mathscr{C}^{\infty}(\mathbb{T}^n; (\dot{\mathscr{C}}^\infty(X))')$ in the series
\begin{equation}\label{FourierExpIntro_TxX}
	u(t) = \sum_{j \in \N^\ast} u_j(t) \phi_j,
\end{equation}
with coefficients $u_j$ belonging to some Gevrey distributions classes on $\T^n$.

Subsequently, such expansions are applied to the study of the global hypoellipticity of operators of the type 
\begin{equation}\label{eq:op-intro}
	\caL= D_t + c(t) P , \quad  D_t = -i\partial_t, t\in\T, 
\end{equation}
where $c\in\caG^s(\T)$ is a complex-valued Gevrey function on $\T$. Namely, it follows from \eqref{FourierExpIntro_TxX} that the equation $\caL u=f$ 
 is equivalent to  sequence of ordinary differential equations,
\begin{equation}\label{eq:odesintro}
	D_tu_j(t) + \lambda_j c(t) u_j(t) = f_j(t), \ t \in \T, j \in \N^\star.
\end{equation}
Hence, according to Theorem \ref{thm:reprF} and its corollaries below, 
we may characterise the regularity of $u$ by analyzing  the behaviour of the solutions $u_j(t)$  (and of their derivatives) as $j \to \infty$.

It is interesting to observe that the behavior of  $\lambda_j$ at infinity and its interactions with the constant
$$
c_0 = \dfrac{1}{2\pi} \int_{0}^{2\pi} c(t) dt
$$
appear as obstructions to the global hypoellipticity of $\caL$ (a phenomenon firstly observed by Hounie in \cite{Hou79}). This is connected with the so-called \textit{Diophantine approximations}, and is closely related to Liouville numbers, a widely explored concept in the study of global properties of operators on the 
torus, as presented, e.g., in \cite{BER,DGY,GW72,Pet05,GPR} and the references quoted therein.

It is worth noting that applications of Fourier expansions for the characterisations of functional spaces are widely employed in the literature. 
For instance, we recall  Seeley's papers \cite{See65,See69}, for the study of  smooth and analytic functions on vector bundles. 
In the case of Hilbert spaces and closed manifolds, we refer, e.g., to Delgado and Ruzhansky \cite{DR}, while for compact Lie groups 
we mention the work by Kirilov, Moraes and Ruzhansky in \cite{KMR2022,KMR2020}. 
On Euclidean spaces, in the context of Gelfand-Shilov classes, we may cite  Cappiello, Gramchev and Pilipovi\'c \cite{GPR,CGPR}. Expansions like \eqref{FourierExpIntro_TxX}  on the product of manifolds (compact or not) and their applications to the study of periodic evolution equation have been recently considered: in the case of $\T \times M$, $M$ a smooth closed manifold, by \'Avila, Gramchev, and Kirilov \cite{AGK18}; in the non-compact case $\T \times \R^d$,
by \'Avila and Cappiello \cite{AC,AC24}, in the setting of Gelfand-Shilov classes, and by \'Avila, Bonino and Coriasco \cite{ABC25}, in the setting of weighted spaces and Schwartz functions and distributions (see also Pedroso Kowacs \cite{PK}).

In this paper we extend, to the more general case $\T\times X$, $X$ an asymptotically Euclidean manifold, the results obtained in \cite{ABC25}, 
focused on the Gevrey-Sobolev-Kato classes on $\T^n\times \R^d$, as well as the similar Fourier analysis, there performed through expansions 
generated by suitable elliptic SG-operators on $\R^d$. Moreover, here we also start the study of the time-dependent coefficient operators 
\eqref{eq:op-intro}, in addition to the time-independent ones, analogous to those appearing in \cite{ABC25}. 

The paper is organized as follows. In Section \ref{sec:spaces} we illustrate the construction of the spaces we are interested in,
and give characterisation results in terms of Fourier expansions. Subsequently, in Section \ref{sec:gh} we study the global
hypoellipticity properties of operators of the form \eqref{eq:op-intro} in the constructed functional setting. For the convenience
of the reader, we have also included a short Appendix, where we recall some basic definitions and properties of the 
asymptotically Euclidean manifolds.

\section{Gevrey-Sobolev-Kato spaces on $\T^n \times X$}\label{sec:spaces}
\setcounter{equation}{0}


We present here the definitions and main properties of the class of Gevrey time-periodic Sobolev-Kato spaces on $\T^n\times X$,
where $X$ is a $d$-dimensional scattering manifold and $\T^n$ is the $n$-dimensional torus. 
By arguments involving admissible local charts and partitions of unity, the construction and the properties
of such spaces are similar to those of the Gevrey time-periodic Sobolev-Kato spaces on $\T^n\times\R^d$ introduced in \cite{ABC25}.

\subsection{Gevrey classes on the torus}

Let us begin by recalling the standard characterisation of the Gevrey classes on the $n$-dimensional torus $\mathbb{T}^n$. Given $\eta > 0$ and $\sigma \geq 1$, define $\mathcal{G}^{\sigma,\eta} = \mathcal{G}^{\sigma,\eta}(\mathbb{T}^n)$ as the space of all smooth functions $u \in \mathscr{C}^\infty(\mathbb{T}^n)$ such that there exists $C = C_{\sigma,\eta} > 0$ for which
$$
\sup_{t \in \mathbb{T}^n} |\partial_t^{\gamma} u(t)| \leq C \eta^{|\gamma|} (\gamma!)^\sigma, \quad \gamma \in \mathbb{Z}_+^n.
$$
This is a Banach space, endowed with the norm
$$
\|u\|_{\mathcal{G}^{\sigma,\eta}} = \sup_{\gamma \in \mathbb{Z}_+^n} \left\{ \sup_{t \in \mathbb{T}^n} \eta^{-|\gamma|} (\gamma!)^{-\sigma} |\partial_t^{\gamma} u(t)| \right\}.
$$
The space of periodic Gevrey functions of order $\sigma$ is then defined by
$$
\mathcal{G}^{\sigma}(\mathbb{T}^n) = \underset{\eta \rightarrow +\infty}{\mathrm{ind}\, \lim} \; \mathcal{G}^{\sigma,\eta}(\mathbb{T}^n).
$$
The dual space $(\mathcal{G}^{\sigma}(\mathbb{T}^n))'$ consists of all linear maps $\theta: \mathcal{G}^{\sigma}(\mathbb{T}^n) \to \mathbb{C}$ such that, for every $\eta > 0$, there exists $C = C_\eta > 0$ satisfying,
for all $u \in \mathcal{G}^{\sigma,\eta}(\T^n)$,
$$
|\langle \theta, u \rangle| \leq C \|u\|_{\mathcal{G}^{\sigma,\eta}}.
$$

\subsection{The spaces $\gh(\T^n \times X)$}

Let $X$ be a $d$-dimensional asymptotically Euclidean manifold. For any fixed $\sigma > 1$, $r, \rho \in \mathbb{R}$, and $C > 0$, we define
$$
\gh^{C} = \gh^{C}(\mathbb{T}^n \times X)
$$
as the space of all functions $u \in \mathscr{C}^{\infty}(\mathbb{T}^n; (\dot{\mathscr{C}}^\infty(X))')$ such that
\begin{equation}\label{eq:GHnorm}
\|u\|_{\gh^C(\mathbb{T}^n \times X)} = \|u\|_{\sigma,C,r,\rho} := 
\sup_{\gamma \in \mathbb{Z}_+^n} \left\{ C^{-|\gamma|} (\gamma!)^{-\sigma} \sup_{t \in \mathbb{T}^n} \| \partial_t^{\gamma} u(t) \|_{H^{r,\rho}(X)} \right\} < \infty.
\end{equation}
If $u \in \gh^{C}$, then:
\begin{itemize}
    \item for any $\gamma \in \mathbb{Z}_+^n$, the map $\mathbb{T}^n \ni t \mapsto \partial_t^{\gamma} u(t) \in H^{r,\rho}(X)$ is well-defined;
    \item the estimate $\sup_{t \in \mathbb{T}^n} \| \partial_t^\gamma u(t) \|_{H^{r,\rho}(X)} \leq C^{|\gamma|+1} (\gamma!)^\sigma$ holds true.
\end{itemize}

\begin{prop}
The spaces $\gh^C(\mathbb{T}^n \times X)$ are Banach spaces endowed with the norm 
$\left\| \cdot \right\|_{\sigma, r,\rho, C}$  given by \eqref{eq:GHnorm}. Moreover, the following inclusions holds true:

\begin{enumerate}
	\item For each $C>0$ and $\sigma \leq \mu$
	\begin{equation}\label{sig<tho}
			\gh^C \subset \mathcal{G}^{\mu}\mathcal{H}^C_{r,\rho};
	\end{equation}

	\item For each $\sigma > 1$ and $C\leq \tilde C$ we have
	\begin{equation}\label{CleqD}
		\gh^C\subset \mathcal{G}^{\sigma}\mathcal{H}^{\tilde C }_{r,\rho};
	\end{equation}
	
	\item For each $\sigma > 1$ and $C>0$ we obtain
	\begin{equation}\label{sig,C,sobolev}
		\mathcal{G}^{\sigma}\mathcal{H}^C_{r,\rho} \subset \mathcal{G}^{\sigma}\mathcal{H}^C_{t,\lambda},
	\end{equation}
	whenever $t \leq r$ and $\lambda \leq \rho$.

\end{enumerate}
\end{prop}

\begin{defn}
With $X$ a $d$-dimensional asymptotically Euclidean manifold, We define the spaces
$$
\mathcal{G}^{\sigma}\mathcal{H}_{r,\rho} (\T^n\times X) = 
\bigcup_{C > 0} \mathcal{G}^{\sigma} \mathcal{H}^C_{r,\rho}(\mathbb{T}^n \times X),
\;
\mathcal{G} \mathcal{H}_{r,\rho} (\T^n\times X) = \bigcup_{\sigma > 1} \mathcal{G}^{\sigma} \mathcal{H}_{r,\rho} (\T^n\times X),
$$ 
endowed with the induced inductive limit topologies.
\end{defn}

\begin{defn}
We define
$$
(\mathcal{G} \mathcal{H}_{r,\rho}')(\T^n\times X)=
\mathcal{G} \mathcal{H}_{r,\rho}' := \left(\mathcal{G} \mathcal{H}_{r,\rho}(\mathbb{T}^n \times X)\right)'
$$
as the space of all linear continuous maps $\theta: \mathcal{G} \mathcal{H}_{r,\rho} \to \mathbb{C}$. Then, a linear functional $\theta: \mathcal{G} \mathcal{H}_{r,\rho}(\mathbb{T}^n \times X) \to \mathbb{C}$ belongs to $\mathcal{G} \mathcal{H}_{r,\rho}'(\mathbb{T}^n \times X)$ if and only if, for every $\sigma > 1$, $C > 0$, there exists $B = B_{\sigma,C} > 0$ such that for every $u \in \mathcal{G}^{\sigma} \mathcal{H}^C_{r,\rho}(\mathbb{T}^n \times X)$,
$$
|\langle \theta, u \rangle| \leq B \sup_{\gamma \in \mathbb{Z}_+^n} \left\{ C^{-|\gamma|} (\gamma!)^{-\sigma} \sup_{t \in \mathbb{T}^n} \| \partial_t^{\gamma} u(t) \|_{H^{r,\rho}(X)} \right\}.
$$

\end{defn}

\begin{prop}
The spaces $\mathcal{G} \mathcal{H}_{r,\rho}$ and $\mathcal{G} \mathcal{H}_{r,\rho}'$ are inductive and projective limit of Banach spaces,
respectively, since
$$
\mathcal{G} \mathcal{H}_{r,\rho}= \bigcup_{\sigma > 1} \mathcal{G}^{\sigma} \mathcal{H}_{r,\rho} = \varinjlim_\sigma \mathcal{G}^{\sigma} \mathcal{H}_{r,\rho}^{\sigma-1}
$$    
and
$$
\mathcal{G} \mathcal{H}_{r,\rho}'= \bigcap_{\sigma > 1} (\mathcal{G}^{\sigma} \mathcal{H}_{r,\rho})' = \varprojlim_\sigma (\mathcal{G}^{\sigma} \mathcal{H}_{r,\rho}^{\sigma-1})'.
$$ 

\end{prop}

\subsection{The spaces $\Fc = \Fc(\mathbb{T}^n \times X)$ and $\Fc' = \left(\Fc(\mathbb{T}^n \times X)\right)^\prime$} 

\begin{defn}
We define the space
\begin{equation}\label{eq:defF}
    \Fc = \Fc(\mathbb{T}^n \times X) \doteq \bigcap_{r, \rho \in \mathbb{R}} \mathcal{G}\mathcal{H}_{r,\rho}(\mathbb{T}^n \times X),
\end{equation}
equipped with the projective limit topology. 
\end{defn}

Specifically, for a sequence $(f_j)_{j\in\N}\subset\Fc$, $f_j \to 0$ as $j \to \infty$ in $\Fc$ if and only if for every pair $(r, \rho) \in \mathbb{R}^2$ there exists $\sigma = \sigma_{r\rho} > 1$ such that
$$
\lim_{j \to \infty} f_j = 0 \quad \text{in} \quad \mathcal{G}^\sigma\mathcal{H}_{r,\rho}.
$$
\begin{defn}
We denote by $\Fc'$ the space of all continuous linear functionals $\theta\colon \Fc \to \mathbb{C}$. Then,
\begin{equation}\label{F'=UH'}
    \Fc' =\Fc'(\mathbb{T}^n \times X)= (\Fc(\mathbb{T}^n \times X))' = \bigcup_{r, \rho \in \mathbb{R}} \mathcal{G}\mathcal{H}_{r,\rho}'(\mathbb{T}^n \times X),
\end{equation}	
equipped with the inductive limit topology.    
\end{defn}

\subsection{Eigenfunction expansions in $\Fc(\T^n \times X)$ and $\Fc'(\T^n \times X)$}

Let $P \in \Psi_{\scc}^{m,\mu}(X)$ be a self-adjoint, elliptic scattering operator of orders $m, \mu > 0$, and let $\{ \phi_j \}_{j \in \mathbb{N}^*} \subset \dot{\mathscr{C}}^\infty(X)$ denote the associated orthonormal basis of eigenfunctions, corresponding to the eigenvalues $\{ \lambda_j \}_{j \in \mathbb{N}^*}$. The sequence $\{\widetilde{\lambda}_j\}_{j\in\N^*}$ denotes the sequence of eigenvalues of $\widetilde{P}=P+P_0$, $P_0$ 
the orthogonal projection on $\ker P$ (that is, it coincides with the
sequence $\{ \lambda_j \}_{j \in \mathbb{N}^*}$, except for the first $N$ vanishing element, 
$N=\dim\ker P<\infty$, which are substituted by $\widetilde{\lambda}_j=1$, $j=1,\dots,N$).

Recall that, if $P$ is, additionally, classical and positive, the asymptotic Weyl law, as studied in \cite{BC11, CorMan13, CD21}, describes the behavior of the spectral counting function as $\lambda \to +\infty$. In this case, we have
$$
N(\lambda) \sim 
\begin{cases}
C_1\, \lambda^{d/\min\{m,\mu\}}, & \quad m \ne \mu, \\
C_2\, \lambda^{d/m} \log \lambda, & \quad m = \mu,
\end{cases}
$$
where $N(\lambda) = N_P(\lambda) = \left| \{ \lambda_j \leq \lambda : \lambda_j \text{ is an eigenvalue of } P \} \right|$ denotes the spectral counting function of $P$, and the constants $C_1, C_2$ depend on the principal symbol of the operator.

The asymptotics of $N(\lambda)$ above also determines the asymptotic behaviour of the eigenvalues $\lambda_j$ (see, e.g., ~\cite{ABC25} for details). Specifically, as $j \to \infty$, one obtains
$$
\lambda_j \sim
\begin{cases}
\widetilde{C}_1\, j^{\min\{m,\mu\}/d}, & \quad m \ne \mu, \\
\widetilde{C}_2 \left( \dfrac{j}{\log j} \right)^{m/d}, & \quad m = \mu,
\end{cases}
$$
for suitable constants $\widetilde{C}_1, \widetilde{C}_2 > 0$.

\begin{rem}\label{rem:polyboundlambda}
Notice that, in such case, there are constants $K,K'>0$ and exponents $\varrho, \varrho' > 0$ such that
$$
K' j^{\varrho'} \leq \lambda_j \leq K j^{\varrho}, \quad j \to \infty.
$$
Moreover, the exponents can be chosen such that $\varrho - \varrho' \leq \epsilon$ for any $\epsilon > 0$.
\end{rem}
\noindent
We first state results concerning eigenfunctions expansions for the weighted Sobolev spaces on $X$. We omit the proofs,
which follow by the same arguments employed to prove the analogous results in \cite{ABC25}, by reduction to the case $X=\R^d$
via admissible local charts and partitions of unity.

\begin{prop}\label{prop:eqnorws}
		Let $P \in \Psi_{\scc}^{m,\mu}(X)$ be an elliptic, normal scattering operator on $X$ with order components $m,\mu>0$. 
		Denote by $P_0$ the orthogonal projection on $\ker P$ and let $\widetilde P=P+P_0$. Then,
		$$
			u \in H^{m,\mu}(X) \Longleftrightarrow P u \in H^{0,0}(X)\equiv L^2(X)
		$$
		and
		$$
			\|u\|_{H^{m,\mu}(X)} \asymp \|\widetilde Pu\|_{L^2(X)}.
		$$
	\end{prop}

 \begin{defn}
    	Let $P \in \Psi_{\scc}^{m,\mu}(X)$ be an elliptic, normal scattering operator on $X$ with order components $m,\mu>0$, 
	and denote by $\{\phi_j\}_{j \in \N^\ast}$ a basis of orthonormal eigenfunctions of $P$. 
	Given $f \in \dot{\mathscr{C}}^\infty(X)$ we set
	\[
    		f_j = (f,\phi_j)_{L^2(X)},\quad j \in \N^\ast,
    	\]
    	which implies $f = \sum_{j \in \N^\ast} f_j \phi_j$. By duality, for $u \in (\dot{\mathscr{C}}^\infty(X))'$ we set 
    	\begin{equation}\label{eq:uj}
    		u_j = u\!\left(\overline{\phi_j}\right)=\langle u , \overline{\phi_j}\rangle,
    	\end{equation}
    	which implies $u = \sum_{j \in \N^\ast} u_j \phi_j$.
	\end{defn}
\begin{thm}\label{thm:exp_Hmmu}
		Let $P \in  \Psi_{\scc}^{m,\mu}(X)$ be an elliptic, normal scattering operator 
		with order components $m,\mu>0$, and denote by $\{\phi_j\}_{j \in \N^\ast}$ a basis of orthonormal 
		eigenfunctions of $P$ with corresponding eigenvalues $\{\lambda_j\}_{j \in \N^\ast}$. Let $r \in \N$. Then, $u \in (\dot{\mathscr{C}}^\infty(X))'$
		belongs to $H^{rm,r\mu}(X)$ if and only if 
		\[
			\sum_{j \in \mathbb N} |u_j|^2 |\lambda_j|^{2r} <\infty,
		\]
		with $u_j$ defined in \eqref{eq:uj}. Moreover, 
		\begin{equation}\label{eq:Hmmuseries}
			\|u\|_{H^{rm,r\mu}(X)}^2 \asymp \sum_{j \in \N^\ast} |u_j|^2 | \widetilde{\lambda}_j|^{2r},
		\end{equation}
		with the eigenvalues $\{\widetilde{\lambda}_j\}_{j \in \N^\ast}$ of $\widetilde{P} = P + P_0$, $P_0$ the projection on $\ker P$.
	\end{thm}
\begin{cor}\label{cor:schw}
	Under the same hypotheses of Theorem \ref{thm:exp_Hmmu}, we have, for $u \in (\dot{\mathscr{C}}^\infty(X))'$,
	\[
		u \in \dot{\mathscr{C}}^\infty(X)\iff \sum_{j \in \N^\ast} |u_j|^2 |\lambda_j|^{2M} <\infty \text{ for any $M\in\N$}.
	\]
\end{cor}
\begin{cor}\label{cor:exp_Hmmu}
	Let $P \in \Psi_{\sccl}^{m,\mu}(X)$ be an elliptic, invertible, self-adjoint, positive, classical scattering operator 
	with order components $m,\mu>0$, and denote by $\{\phi_j\}_{j \in \N^\ast}$ a basis of orthonormal 
	eigenfunctions of $P$ with corresponding eigenvalues $\{\lambda_j\}_{j \in \N^\ast}$. Let $r \in \R$. Then, $u \in (\dot{\mathscr{C}}^\infty(X))'$
	belongs to $H^{rm,r\mu}(X)$ if and only if 
	\[
		\begin{cases}
			\displaystyle\sum_{j \in \mathbb N} |u_j|^2 \, j^\frac{2r\min\{m,\mu\}}{d} <\infty, & \text{ if $m\not=\mu$,}
			\\
			\displaystyle\sum_{j \in \mathbb N} |u_j|^2 \left(\frac{j}{\log j}\right)^\frac{2rm}{d} <\infty, & \text{ if $m=\mu$,}
		\end{cases}
	\]
	with $u_j$ defined in \eqref{eq:uj}. Moreover, 
	\begin{equation}\label{eq:Hmmuseriesbis}
		\|u\|_{H^{rm,r\mu}(X)}^2 \asymp \sum_{j \in \N^\ast} |u_j|^2 | {\lambda}_j|^{2r} \asymp
				\begin{cases}
			\displaystyle\sum_{j \in \mathbb N} |u_j|^2 \, j^\frac{2r\min\{m,\mu\}}{d} <\infty,  &\text{if $m\not=\mu$,}
			\\
			\displaystyle\sum_{j \in \mathbb N} |u_j|^2 \left(\frac{j}{\log j}\right)^\frac{2rm}{d} <\infty,  &\text{if $m=\mu$.}
		\end{cases}
	\end{equation}
\end{cor}
\begin{thm}\label{thm:unifconver}
	Let $P \in \Psi_{\sccl}^{m,\mu}(X)$ be an elliptic, normal, classical scattering operator 
	with order components $m,\mu>0$ or $m,\mu<0$, and denote by $\{\phi_j\}_{j \in \N^\ast}$ a basis of orthonormal 
	eigenfunctions of $P$. Then, $f\in\dot{\mathscr{C}}^\infty(X)$ if and only if 
	\begin{equation}\label{eq:convunsch}
		\sum_{j \in \N^\ast} \left|(f,\phi_j)_{L^2(X)} \right| \, \left|(A\phi_j)(x) \right|, \quad x\in X.
	\end{equation}
	converges uniformly on $X$ for every scattering operator $A$.
\end{thm}

Let $f \in \Fc$. Then, for all $r, \rho \in \R$ there exist $\sigma = \sigma_{r\rho} > 1$, $C = C_{r\rho} > 0$, such that $\|f\|_{\sigma,C, r,\rho} <\infty$,
since $\dot{\mathscr{C}}^\infty(X)= \bigcap_{r,\rho \in \R} H^{r,\rho}(X)$
(as sets and topologically).
We define the Fourier coefficients of $f(t)$, $f\in\Fc(\T^n\times X)$, with respect to the eigenfunction basis $\{\phi_j\}$, as above, namely,
$$
f_j(t) \doteq (f(t), \phi_j)_{L^2(X)}, \quad j \in \N^\ast.
$$
Then, for every multi-index $\gamma \in \Z_+^n$, we have
$$
\partial_t^\gamma f_j(t) = \left( \partial_t^\gamma f(t), \phi_j \right)_{L^2(X)},
$$
hence
$$
\partial_t^\gamma f(t) = \sum_{j \in \N^\ast} \partial_t^\gamma f_j(t) \phi_j.
$$
We now state the main results about the corresponding Fourier expansions in $\Fc$ and $\Fc'$. We again omit the proofs, 
since they follow by arguments and computations analogous to those presented in \cite{ABC25} for the case $X=\R^d$.
\begin{thm}\label{thm:reprF}
Let $P \in \Psi_{\scc}^{m,\mu}(X)$ be a self-adjoint, elliptic scattering operator with order components $m, \mu>0$. Then $f \in \Fc(\T^n \times X)$ if and only if
$$
f(t) = \sum_{j \in \N^\ast} f_j(t) \phi_j,
$$
with $f_j(t)$ defined as above and satisfying the condition: 
\begin{equation}\label{eq:s}
			\tag{*}
			\hspace*{-3mm}
			\begin{aligned}
				\text{for }& \text{every $M \in \N$ there exist $\sigma\!=\!\sigma_M\!>\!1$ and $C\!=\!C_M\!>\!0$ such that}
				\\
				&\sup_{t \in \T^n}  \sum_{j \in \N^\ast} \widetilde{\lambda}_j^{2M} |\partial_t^\gamma f_j(t)|^2
				=
				\sup_{t \in \T^n}  \sum_{j \in \N^\ast} \widetilde{\lambda}_j^{2M} |f^\gamma_j(t)|^2 \leq C^{2(|\gamma|+1)} (\gamma!)^{2 \sigma} ,
				\\
				\text{for }&\text{every $\gamma \in \Z_+^n$}.
			\end{aligned}
		\end{equation}
\end{thm}

\begin{cor}\label{cor:estpow}
		Let $P \in \Psi_{\sccl}^{m,\mu}(X)$ be a positive, self-adjoint, classical, elliptic scattering operator with order components $m, \mu>0$.    
 		In this case, condition \eqref{eq:s} in Theorem \ref{thm:reprF} is equivalent to the condition
		\begin{equation}\label{eq:sclsa}
			\tag{**}
			\hspace*{-3mm}
			\begin{aligned}
				\text{for every  } &\text{$M \in \N$ there exist $\sigma\!=\!\sigma_M\!>\!1$ and $C\!=\!C_M\!>\!0$ such that}
				\\
				&\sup_{t \in \T^n} |\partial^\gamma_t f_j(t)| =
				\sup_{t \in \T^n} | f_j^\gamma(t)| \leq C^{|\gamma|+1} (\gamma!)^{\sigma} |\widetilde{\lambda}_j|^{-M},
				\\
				\text{for every } & j \in \N^\ast, \gamma \in \Z_+^n.
			\end{aligned}
		\end{equation}
	\end{cor}

	\begin{cor}\label{cor:seq}
		Assume that $f \in \Fc(\T ^n\times X)$. 
		\begin{itemize} 
			\item[ i)] Under the hypotheses of Theorem \ref{thm:reprF},
				for every $M \in \N$ there exist $\sigma=\sigma_M>1$ and $C=C_M>0$ such that
				\begin{equation}\label{eq-seq}
					\|f_j\|_{\mathcal{G}^{\sigma,C}(\T^n)} \leq C |\widetilde{\lambda}_j|^{-M}, \quad j \in \N^\ast.
				\end{equation}
			\item[ii)] Under the hypotheses of Corollary \ref{cor:estpow},  
				for every $M \in \N$ there exist $\sigma=\sigma_M>1$ and $C=C_M>0$ such that
				\[
					\|f_j\|_{\mathcal{G}^{\sigma,C}(\T^n)} \leq C j^{-M\varrho}, \ j \to \infty,
				\]
				for some $\varrho>0$. In particular, $\{f_j\}_{j \in \N^\ast} \subset \mathcal{G}^{\sigma_1}(\T^n)$. 
		\end{itemize}
\end{cor}

	\begin{cor}\label{cor:suffhypo}
		Under the hypotheses of Corollary \ref{cor:estpow}, let $\{f_j\}_{j \in \N^\ast}\subset\mathcal{G}^{\sigma,C}(\T^n)$ 
		be a sequence with the property that for every $M' \in \N$ there exists $B=B_{M'}>0$ such that
		\[
			\|f_j\|_{\mathcal{G}^{\sigma,C}(\T^n)} \leq B j^{-M'}, \quad j \to \infty.
		\]
		Then, setting $f(t)=\sum_{j \in \N^\ast} f_j (t) \phi_j$, $t\in\T^n$, it follows $f \in \Fc(\T ^n\times X)$.
	\end{cor}

	\begin{thm}\label{thm:unifconvert}
	Let $P \in \Psi_{\scc}^{m,\mu}(X)$ be a self-adjoint, elliptic scattering operator with order components $m, \mu>0$ or $m, \mu<0$,
		and let $\{\phi_j\}$ be a corresponding orthonormal basis of eigenfunctions. 
		\begin{enumerate}
		\item\label{point:implltr} For any $f \in \Fc$ the series 
		\begin{equation}\label{eq:convunscht}
			\sum_{j \in \mathbb N} \left|\partial^\gamma_t(f(t),\phi_j)_{L^2(\R^d)} \right| \, |A\phi_j(x)| , \quad \gamma \in \Z_+^n,
		\end{equation}
		converge uniformly on $\T^n\times X$ for every scattering operator $A$ and there exist $B=B_{APnd}>0$, 
		$\sigma=\sigma_{APnd}>1$, $C=C_{APnd}>0$, depending only on $A,P,n$, and $d$,
		such that the sums $S_{\gamma A}$ of \eqref{eq:convunscht} satisfy the condition
		\begin{equation}\label{eq:sumconvunift}
			S_{\gamma A}(t,x)\le B C^{|\gamma|+1} (\gamma!)^{\sigma}, \quad t\in\T^n, x\in X.
		\end{equation}
		\item If we additionally assume that $P \in \Psi_{\sccl}^{m,\mu}(X)$, $m,\mu>0$, is positive and classical, then, for $f\in\cinf(\T^n,L^2(X))$ the converse of the implication
		in point \eqref{point:implltr} above holds true as well.
		\end{enumerate} 
	\end{thm}

Consider now $\theta \in \Fc'$ and  $M \in \Z$ such that $\theta \in \mathcal{G}\mathcal{H}_{Mm,M\mu}'$. 
	For any $\psi \in \mathcal{G}^{\sigma}(\T^n)$ we consider $\psi \otimes \overline{\phi_j} \in \mathcal{G}^{\sigma}\mathcal{H}_{Mm,M\mu}$ by setting
	\begin{align*}
		\T^n \ni t \mapsto \psi(t)\overline{\phi_j} \colon  \R^d \to \C \colon x \mapsto \psi(t)\overline{\phi_j(x)}, \quad j \in \N^\ast.
	\end{align*}
	Then, there are well-defined linear maps $\theta_j: \mathcal{G}^{\sigma}(\T^n) \to \C, $ given by
	\begin{equation}\label{partial_in_F'}
		\langle  \theta_j  ,  \psi \rangle \doteq 
		\langle  \theta ,  \psi \otimes \overline{\phi_j}  \rangle, \quad j \in \N^\ast.
	\end{equation}
	We claim that  $\theta_j \in (\mathcal{G}^{\sigma}(\T^n))'$. 
	Indeed, given any constant $C>0$ there exists $B=B_{\sigma C}>0$ 
	such that
	\begin{equation}\label{eq:thetaseriesest}
		\begin{aligned}
			|\langle  \theta_j  ,  \psi  \rangle| & = 
			|\langle  \theta  ,  \psi \otimes \overline{\phi_j}  \rangle|  \\
			&\leq B
			\sup_{\gamma \in \Z_{+}^n}\left\{
			C^{-|\gamma|}(\gamma!)^{-\sigma} \sup_{t \in \T^n} \|\partial_t^{\gamma} \psi(t)\overline{\phi_j}\|_{H^{Mm,M\mu}}
			\right\} \\
			& = B \|\overline{\phi_j}\|_{H^{Mm,M\mu}}
			\sup_{\gamma \in \Z_{+}^n}\left\{
			C^{-|\gamma|}(\gamma!)^{-\sigma} \sup_{t \in \T^n} |\partial_t^{\gamma}\psi(t)|\right\}
			\\
			& = B \|\phi_j\|_{H^{Mm,M\mu}}
			\|\psi\|_{\mathcal{G}^{\sigma,C}} \le \widetilde{B} |\widetilde{\lambda}_j|^M \|\psi\|_{\mathcal{G}^{\sigma,C}}, \quad j \in \N^\ast,
	\end{aligned} 
	\end{equation}
	which proves the assertion. This allows us to decompose any $\theta\in\Fc'$ into a series of tensor products,
	whose first factors satisfy the estimates \eqref{eq:thetaseriesest}, as shown in the subsequent Lemma \ref{lem:thetaseries}. 
	Next, we state the convergence result in $\Fc'$.
    
	\begin{lemma}\label{lem:thetaseries}
		Let $\theta \in \Fc'(\T ^n\times X)$. Then
		\begin{equation}\label{eq:defserie}
			\theta=\sum_{j \in \N^\ast}\theta_j\otimes\phi_j,
		\end{equation}
		with $\{\theta_j\}_{j \in \N^\ast} \subset (\mathcal{G}^{\sigma}(\T^n))'$ given by \eqref{partial_in_F'}, and so satisfying \eqref{eq:thetaseriesest}.
	\end{lemma}

	\begin{thm}\label{thm:cauchyseq}
		Let $\{\tau_j\}_{j \in \N^\ast}\subset\Fc'(\T ^n\times X)$ be such that
		$\{\langle \tau_j,f \rangle \}_{j \in \N^\ast}$ is a Cauchy sequence in $\C$, for all 
		$f \in \Fc(\T ^n\times X)$. Then there exists $\tau \in \Fc'(\T ^n\times X)$ such that
		$\displaystyle\tau = \lim_{j \to \infty}  \tau_j$, that is, 
		\[
			\langle \tau, f \rangle = \lim_{j \to \infty}  \langle \tau_j, f \rangle, \quad f \in \Fc(\T ^n\times X).
		\]
	\end{thm}

	The subsequent Theorem \ref{thm:seq} provides a sufficient condition on the coefficients of an expansion with
	respect to the basis $\{\phi_j\}_{j \in \N^\ast}$ to converge to an element of $\Fc'$. Together with Lemma \ref{lem:thetaseries}
	shows the characterisation of $\Fc'$ in terms of eigenfunctions expansions associated with a classical,
	self-adjoint, positive scattering operator.
	
	\begin{thm}\label{thm:seq}
		Let $P \in \Psi_{\sccl}^{m,\mu}(X)$ be an elliptic, self-adjoint, positive, 
		classical scattering operator with order components $m,\mu>0$, and denote by $\{\phi_j\}_{j \in \N^\ast}$ a basis 
		of orthonormal eigenfunctions of $P$ with corresponding eigenvalues $\{\lambda_j\}_{j \in \N^\ast}$.
		Let $\{\vartheta_j\}_{j \in \N^\ast} \subset \bigcap_{\sigma>1}(\mathcal{G}^{\sigma}(\T^n))'$ 
		be a sequence such that there exist $M \in \Z$,  $B>0$, satisfying 
		\[
			|\langle  \vartheta_j ,  \psi \rangle| \leq B \|\psi\|_{\mathcal{G}^{\sigma,C}(\T^n)}
			|\widetilde{\lambda}_{j}|^{M}, \quad j \in\N^\ast, 
		\]
		for all $\sigma>1$, $C>0$, $\psi \in \mathcal{G}^{\sigma,C}(\T^n)$. Then
		\begin{equation}\label{fourier-inv}
			\vartheta = \sum_{j \in \N^\ast} \vartheta_j \otimes\phi_j\in \Fc'(\T ^n\times X).
		\end{equation}
		Moreover, 
		\[
			\langle  \vartheta_j , \psi \rangle  = \langle  \vartheta ,  \psi\otimes\overline{\phi_j} \rangle, \quad \psi \in  \mathcal{G}^{\sigma}(\T^n),j \in \N^\ast.
		\]
	\end{thm}

\section{Global hypoellipticity for time-periodic operators on $\T \times X$}\label{sec:gh}
\setcounter{equation}{0}

We first state results about the global hypoellipticity of operators of the type
\begin{equation}\label{eq:L_constant_coeff}
	L = D_t + \omega P, \ t \in \T,
\end{equation}
defined on (suitable) distributions on $\T\times X$. Here $\omega = \alpha + i\beta \in \C$ and 
$P\in\Psi_\mathrm{sc}^{m,\mu}(X)$ is a classical, positive, self-adjoint, elliptic scattering operator on $X$ with order components $m, \mu>0$.
We will omit the proofs, since they are consequences of the computations performed in \cite{ABC25}, Sections 4 and 5. 

The strategy consists in studying the solutions $u\in\Fc' (\T \times X)$ of the equation $Lu =f$ employing a Fourier decomposition of both $u\in\Fc'(\T \times X)$ and $f\in\Fc(\T \times X)$ in terms of eingenfunctions generated by the operator $P$, as illustrated in the previous Section \ref{sec:spaces}. 
Therefore, writing, for $t\in\T$,
\begin{equation*}
	u(t) = \sum_{j \in \N^\ast} u_j(t) \phi_j =  \sum_{j \in \N^\ast} u_j(t) \otimes \phi_j
	\ \textrm{ and } \
	f(t) = \sum_{j \in \N^\ast} f_j(t) \phi_j = \sum_{j \in \N^\ast} f_j(t) \otimes \phi_j,
\end{equation*}
we obtain that the equation $Lu=f$ is equivalent to the (infinite) system of ODEs
\begin{equation}\label{diffe-equations}
	D_t u_j(t) +  \omega\lambda_j  u_j(t) = f_j(t), \ t \in \T, \ j \in \N^\ast,
\end{equation}
whose solutions are given by
\begin{equation}\label{fator_integrante}
	u_j(t) = u_{j0}\exp\left( -i \lambda_j\omega t \right) +  i\int_{0}^{t}\exp\left(i\lambda_j \omega (s-t) \right) f_j(s)ds,
\end{equation}
for some $u_{j0} \in \C$, $j \in \N^\ast$. 

We may assume that the coefficients  $f_j$ belong to a Gevrey class $\mathscr{G}^{\sigma}=\mathscr{G}^{\sigma}(\T)$ for all $j \in \N^\ast$. Then, 
by the properties of equation \eqref{diffe-equations}, also $u_j$ is in $\mathscr{G}^{\sigma}$ for all $j \in \N^\ast$. 

We recall that the set 
\begin{equation}\label{eq:defZ}
 \mathcal Z = \{j \in \N^\ast; \ \omega \lambda_{j} \in \Z\}
\end{equation}
is finite if $\beta \neq 0$.
Therefore, if $j \notin \mathcal Z$, $u_{j0}$ is uniquely defined and \eqref{fator_integrante} can be written in either of the two equivalent forms
\begin{align}
	\label{Solu-1-Constant}
	u_j(t) &= \frac{i}{1 - e^{-  2 \pi i\lambda_j\omega}} \int_{0}^{2\pi}\exp\left(-i\lambda_j \omega s\right) f_j(t-s)ds
\intertext{or}
	\label{Solu-2-Constant}
	u_j(t) &= \frac{i}{e^{ 2 \pi i \lambda_j \omega} - 1} \int_{0}^{2\pi}\exp\left(i\lambda_j \omega s \right) f_j(t+s)ds. 
\end{align}
\begin{defn}\label{def:gh}
	We say that the operator $L$ defined in \eqref{eq:L_constant_coeff} is globally hypoelliptic on $\T\times X$ if
	\begin{equation}\label{eq:defgh}
		u \in \Fc'(\T\times X) \ \textrm{ and } \ Lu \in \Fc(\T\times X)
		\Rightarrow
		u \in \Fc(\T\times X).
	\end{equation}
\end{defn}
The results are closely related to a so-called \emph{Diophantine condition}, expressed in the next Definition \ref{def:condA}.
\begin{defn}[Condition \condA]\label{def:condA}
We say that  a real number $\alpha$ satisfies Condition \condA\ if there are positive constants $\epsilon$ and  $C$  such that 
\[
	|\tau -\alpha \lambda_{j}| \geq C j^{-\epsilon}, 
\]
for all $(j,\tau) \in \N^\ast \times \Z$.
\end{defn}
In the literature, Condition \condA\ is also known as the property for the real number $\alpha$ of \emph{not being Liouville} 
with respect to the sequence $\{\lambda_j\}$. 
	The characterisation of the global hypoellipticity on $\T\times X$ of the operator $L$ in \eqref{eq:L_constant_coeff} 
	is expressed in the next Theorem \ref{thm:MainTheormHypo}.
	\begin{thm}\label{thm:MainTheormHypo}
		The operator $L$ defined in \eqref{eq:L_constant_coeff} is globally hypoelliptic on $\T\times X$
		if and only if either $\beta \neq 0$ or $\beta=0$ and $\alpha$ satisfies Condition \condA.
	\end{thm}
	
	\begin{rem}\label{rem:solv}
		The solvability results for the operator $L$ in \eqref{eq:L_constant_coeff} on $\T\times\R^d$ proved in \cite{ABC25}
		extend to the $\T\times X$ setting as well.
	\end{rem}
	
%

Let us now discuss the global hypoellipticity on $\T\times X$ of operators of the type 
\begin{equation}\label{L_timedep_coeff}
	\mathcal L = D_t + c(t) P, \ t \in \T,
\end{equation}
where $c(t) = a(t) + ib(t)$ is a complex-valued function belonging to the Gevrey class $\mathcal{G}^{s}(\T)$, $s >1$, 
and $P \in \Psi_{\sccl}^{m,\mu}(X)$ is a positive, self-adjoint, elliptic, classical scattering operator with order components $m, \mu>0$. 
We say that the operator $\caL$ is globally hypoellitptic on $\T\times X$ if \eqref{eq:defgh} holds true with $\caL$ in place of $L$.

Arguing by Fourier decomposition, as discussed above, assuming that
\begin{align}
	\label{eq:uexp}
	u(t) = \sum_{j \in \N^\ast} u_j(t) \phi_j =  \sum_{j \in \N^\ast} u_j(t) \otimes \phi_j
	\ \intertext{ and } \
	\label{eq:fexp}	
	f(t) = \sum_{j \in \N^\ast} f_j(t) \phi_j = \sum_{j \in \N^\ast} f_j(t) \otimes \phi_j,
\end{align}
the equation $\mathcal Lu=f$ is equivalent to the infinite system of ODEs
\begin{equation}\label{diffe-equations-t}
	D_t u_j(t) +  c(t)\lambda_j  u_j(t) = f_j(t), \ t \in \T, \ j \in \N^\ast.
\end{equation}
Notice that, if we denote by $c_0$ the average on $\T$ of the function $c(t)$, that is, 
$$
c_0 = (2 \pi)^{-1} \int_{0}^{2 \pi} c(t) dt \doteq a_0 + ib_0,
$$
we observe that \eqref{diffe-equations-t} ensures the following result.
\begin{prop} \label{solutions-t}
	Let $u$ and $f = \caL u$ be as in \eqref{eq:uexp} and \eqref{eq:fexp}, respectively.
	Then, for each $j\in\N^\ast$ such that $\lambda_jc_0 \notin \Z$,  \eqref{diffe-equations-t} has a unique solution, 
	which can be written in the following, equivalent two ways:
	\begin{align}\label{Solu-1-t}
		u_j(t) = \frac{i}{1 - e^{-  2 \pi i\lambda_j c_0}} \int_{0}^{2\pi}\exp\left(-i\lambda_j\int_{t-\zeta}^{t}c(r) \, dr\right) f_j(t-\zeta)d\zeta, 
	\end{align}
	or
	\begin{align}\label{Solu-2-t}
		u_j(t) = \frac{i}{e^{ 2 \pi i \lambda_j c_0} - 1} \int_{0}^{2\pi}\exp\left(i\lambda_j\int_{t}^{t+\zeta}c(r) \, dr\right) f_j(t+\zeta)d\zeta. 
	\end{align}
\end{prop}
In the next two lemmas we recall, respectively, a standard formula and a result which is often useful in this kind of analysis.
\begin{lemma}\label{lem:expA}
	Let $p, q$ be positive numbers and $\tau \in \Z_+$. For every  $\mu>0$ there exist $C=C(\mu, p/q)>0$ such that, for all $A \in [0,+\infty)$,
	\begin{equation*}
		A^{\tau p} \exp\left(-\mu A^{q}\right) \leq   C^\tau (\tau!)^{p/q}.
	\end{equation*}
\end{lemma}		
\begin{proof}
	The claim follows by elementary estimates and the behaviour of the function $f(A)=A^{\tau p} \exp\left(-\mu A^{q}\right)$ for $A \in [0,+\infty)$.
\end{proof}
\begin{lemma}\label{lem:1-exp}
	Let  $\{\beta_j \}_{j \in \N^\ast}$ be a sequence of real numbers. Then, for each $j \in \N^\ast$ there exist $l(j) \in \Z$ such that
	\[
		|1 - e^{2\pi i \beta_j}| \geqslant 4 \ | \beta_j + l(j)|.
	\]
\end{lemma} 

\begin{proof}
	See Proposition 5.7 in \cite{AGK18}.
\end{proof}
We can now state results that provide sufficient and necessary conditions for the global hypoellipticity of $\caL$,
involving the change of sign properties of $b$.

\begin{thm}\label{suff1}
Suppose that either $b<0$, or $b>0$, that is, $b$ does not change sign and $b \neq 0$ on $[0,2\pi]$. Then, $\mathcal L$ is globally hypoelliptic. 
\end{thm}
\begin{thm}\label{nec1}
Suppose that $b$ is not identically zero on any subinterval in $[0,2\pi]$. If $b$ changes sign, $\caL$ is not globally hypoelliptic.
\end{thm}
\begin{proof}[Proof of Theorem \ref{suff1}]
Let us assume $b<0$ and consider the solutions of \eqref{diffe-equations-t} as given by \eqref{Solu-1-t}. The argument for $b>0$ is completely similar,
employing instead \eqref{Solu-2-t}, and is then omitted. Since 
\[
	\Theta_j = \left|\dfrac{i}{1 - e^{-  2 \pi i\lambda_j c_0}}\right|\to 1,
\]
 we may write, for any $\gamma\in\Z_+$, $t\in[0,2\pi]$,
\[
	|\partial_t^{\gamma} u_j(t)| \leq C \sum_{\ell = 0}^{\gamma}\binom{\gamma}{\ell} 
	\int_{0}^{2\pi}  \left|\partial_t^{\ell}\exp \left(-i\lambda_{j} \int_{t-\zeta}^{t}c(r)dr\right)\right| |\partial_t^{\gamma-\ell}f_{j}(t-\zeta)|d\zeta.
\]
	
Let us set, for convenience,
\[
	\caH_j(t,\zeta) = \exp \left(-i\lambda_{j} \int_{t-\zeta}^{t}c(r)dr\right)
	\Rightarrow
	|\caH_j(t,\zeta)|=\exp \left(\lambda_{j} \int_{t-\zeta}^{t}b(r)dr\right),
\]
and
\[
	-\theta = \max_{r\in[0,2\pi]} b(r) < 0, \, -\vartheta = \min_{r\in[0,2\pi]} b(r) < 0\Rightarrow -\vartheta\le b(r) \le -\theta, r\in[0,2\pi].
\]
Then, for any $\mu>0$, to be fixed later,
\begin{equation}\label{eq:intexp}
\begin{aligned}
		\int_{0}^{2\pi}\exp\left(\mu \lambda_{j} \int_{t-\zeta}^{t}b(r)dr\right)d\zeta & \leqslant \int_{0}^{2\pi}\exp\left(-\mu  \lambda_{j} \theta\zeta \right)d\zeta
        \\&= \dfrac{1}{\mu  \lambda_{j}\theta}\left(1 - \exp(-2\pi\mu  \lambda_{j}\theta)\right) \\ 
        & \leq 
	C_1 \lambda_j^{-1}\le C_2 j^{-\varrho'},	
\end{aligned}
\end{equation}
since $\lambda_{j} \geq C'j^{\varrho'}$. We also have, for any $k\in\N$ and $\zeta\in(0,2\pi]$,
\[
	\left|\int_{t-\zeta}^tb(r)\,dr\right|^{-k}=|b(r_{t,\zeta})\zeta|^{-k}\le\theta^{-k}\zeta^{-k}.
\]

\smallskip

By the Fa\`a di Bruno formula, for $\ell\ge1$,
\[
	\partial_t^\ell  \mathcal{H}_j(t,\zeta) = \sum_{\Delta(k), \, \ell}
	\frac{(-i\lambda_j)^k}{k!}
	\frac{\ell!}{\ell_1! \cdots \ell_k! }
	\left(\prod_{\nu=1}^k
	\partial_t^{\ell_\nu-1}(c(t)-c(t-\zeta)) \right)
	\mathcal{H}_j(t,\zeta),
\]
where we have employed the notation
	$
	\sum\limits_{\Delta(k), \, \ell} = \sum\limits_{k=1}^\ell\sum\limits_{\stackrel{\ell_1+\ldots+\ell_k=\ell}{\ell_\nu \geq 1, \nu=1,\dots,k}}.
	$	
Since $c\in\caG^s(\T)$, by the Mean Value Theorem, for any $k\in\N^\ast$ and $\zeta\in[0,2\pi]$,
\[
	\left| \prod_{\nu=1}^k
	\partial_t^{\ell_\nu-1}(c(t)-c(t-\zeta))  \right| = 
	\prod_{\nu=1}^k |(\partial_t^{\ell_\nu}c)(t_{\nu t \zeta})\cdot\zeta| 
	\leq \zeta^k C_3^{\ell}\left[\prod_{\nu=1}^k \ell_\nu!\right]^s.
\]

It follows, by Lemma \ref{lem:expA} with $A=\displaystyle-\lambda_j\int_{t-\zeta}^tb(r)\,dr$, $\tau=k$, $p=q=1$, $\mu=1/2$,
for any $j,\ell\in\N^\ast$, $\zeta\in(0,2\pi]$,
\begin{align*}
	|\partial_t^\ell  \mathcal{H}_j(t,\zeta)| &\leq 
	\sum_{\Delta(k), \, \ell}
	\frac{\lambda_j^k}{k!}
	\frac{\ell!}{\ell_1! \cdots \ell_k! }
	\zeta^k C_3^{\ell}\left[\prod_{\nu=1}^k \ell_\nu!\right]^s
	|\caH_j(t,\zeta)|
	\\
	&=C_3^{\ell}\sum_{\Delta(k), \, \ell}
	\frac{1}{k!}
	\frac{\ell!}{\ell_1! \cdots \ell_k! }
	\zeta^k \left|\int_{t-\zeta}^tb(r)\,dr\right|^{-k}
	\left[\prod_{\nu=1}^k \ell_\nu!\right]^s\cdot
	\\
	&\phantom{=C_3^{\ell}\sum_{\Delta(k), \, \ell}}
	\cdot
	\left|\left(-\lambda_j\int_{t-\zeta}^tb(r)\,dr\right)^{k}|\caH_j(t,\zeta)|^\frac{1}{2}\right| \, |\caH_j(t,\zeta)|^\frac{1}{2}
	\\
	&\le C_3^{\ell}\sum_{\Delta(k), \, \ell}
	\frac{\ell!}{k!}
	\theta^{-k}
	\left[\prod_{\nu=1}^k \ell_\nu!\right]^{s-1} C_4^k k! \, |\caH_j(t,\zeta)|^\frac{1}{2}
	\\
	&\le C_5^{\ell} \,\ell!\sum_{\Delta(k), \, \ell}	\left[\prod_{\nu=1}^k \ell_\nu!\right]^{s-1} |\caH_j(t,\zeta)|^\frac{1}{2}.
\end{align*}
Now, recalling \eqref{eq:intexp}, that $f\in\Fc$ implies that $f_j$, $j\in\N^\star$, satisfies Condition \eqref{eq:sclsa} in Corollary \ref{cor:estpow},
and that the estimate above holds a.e. $\zeta\in[0,2\pi]$, setting $\widetilde{\sigma}_M=\max\{s,\sigma_M\}$, we find, for any $j,M,\gamma\in\N^\ast$,
\begin{align*}
	|\partial_t^{\gamma} u_j(t)| &\leq C \left\{\sum_{\ell = 1}^{\gamma}\frac{\gamma!}{\ell!\,(\gamma-\ell)!}
	C_5^{\ell} \,\ell!\sum_{\Delta(k), \, \ell}	\left[\prod_{\nu=1}^k \ell_\nu!\right]^{s-1} 
	C_M^{\gamma-\ell+1}[(\gamma-\ell)!]^{\sigma_M}\lambda_j^{-M}\cdot\right.
	\\
	&\cdot\int_{0}^{2\pi} |\caH_j(t,\zeta)|^\frac{1}{2}\,d\zeta
	\\
	&+C_M^{\gamma+1}(\gamma!)^{\sigma_M}\lambda_j^{-M}\int_{0}^{2\pi} |\caH_j(t,\zeta)|\,d\zeta
	\left.\rule{0mm}{9mm}\right\}
	\\
	&\le \widetilde{C}_M^{\gamma+1}\lambda_j^{-M-1}(\gamma!)^{\widetilde{\sigma}_M}\left\{
	1+\sum_{\ell = 1}^{\gamma}\left[\frac{(\gamma-\ell)!}{\gamma!}\right]^{\widetilde{\sigma}_M-1}
	\sum_{\Delta(k), \, \ell}\left[\prod_{\nu=1}^k \ell_\nu!\right]^{\widetilde{\sigma}_M-1}
	\right\}.
\end{align*}
Observing that
\begin{align*}
	\sum_{\ell = 1}^{\gamma}\left[\frac{(\gamma-\ell)!}{\gamma!}\right]^{\widetilde{\sigma}_M-1}&\sum_{\Delta(k), \, \ell}	
	\left[\prod_{\nu=1}^k \ell_\nu!\right]^{\widetilde{\sigma}_M-1}
	=
	\sum_{\ell = 1}^{\gamma}\left[\frac{\ell!(\gamma-\ell)!}{\gamma!}\right]^{\widetilde{\sigma}_M-1}
	\sum_{\Delta(k), \, \ell}\left[\frac{\prod_{\nu=1}^k \ell_\nu!}{\ell!}\right]^{\widetilde{\sigma}_M-1}
	\\
	\le\sum_{\ell = 1}^{\gamma}&\sum_{\Delta(k), \, \ell}1 = 
	\sum_{\ell = 1}^{\gamma}\sum_{k=1}^\ell\begin{pmatrix}\ell-1\\k-1\end{pmatrix}=\sum_{\ell = 1}^{\gamma}2^{\ell-1}=\frac{2^\gamma-1}{2-1}<2^\gamma,
\end{align*}
we conclude that, for any $j,M,\gamma\in\N^\ast$, for different $C_M>0$, $\sigma_M>1$,
\begin{equation}\label{eq:ujFest}
	\lambda_{j}^{M+1} \sup_{t \in \T}|\partial_t^{\gamma} u_j(t)|\leq C_{M}^{\gamma +1} (\gamma!)^{\sigma_M}.
\end{equation}
In view of \eqref{eq:intexp}, we see that the estimate \eqref{eq:ujFest} holds true also for $\gamma=0$.
By Corollary \ref{cor:estpow}, it follows $u\in\Fc$. The proof is complete.
\end{proof}
We investigate now the effect of a change of sign condition on the coefficient $b(t)$, namely, by admitting the existence of 
$t^+,t^{-}\in[0,2\pi]$ such that
\[
b(t^+)>0 \ \textrm{ and } \ b(t^-)<0.
\]
This provides the necessary condition for the global hypoellipticity stated in Theorem \ref{nec1}.

In order to tackle this problem, we define, for each $\eta \in [0,2\pi]$, the function $\mathcal{B}_{\eta}: [0,2\pi] \to \R$ given by
\[
	\mathcal{B}_{\eta}(t) = \int_{\eta}^{t}b(s)\,ds.
\]	
\begin{lemma}\label{lem:bchngsign}
	Let $b$ be a smooth, real-valued, $2\pi$-periodic function on $\R$, such that $b \not \equiv 0$ on any interval. 
	Then, the following properties are equivalent:
	\begin{enumerate}
		\item[a)] $b$ changes sign;
		\item[b)] there exists $t_0\in \R$ and $t^*, t_* \in (t_0, t_0+2\pi)$ such that
		\[
			\forall t \in (t_0, t_0 + 2\pi] \; \mathcal{B}_{t^*}(t)  \leqslant   0 
			\text{ and } 
			\forall t \in (t_0, t_0 + 2\pi[ \; \mathcal{B}_{t_*}(t) \geqslant   0;
		\]
		\item[c)] there exists $t_0 \in \R$, partitions
		\begin{align*}
			& t_0 < \alpha^* < \gamma^* < t^* <\delta^* <\beta^* < t_0 +2\pi,  \\[1mm]
			& t_0 < \alpha_* < \gamma_* < t_* <\delta_* <\beta_* < t_0 +2\pi,
		\end{align*}
		and  positive constants $c^*, c_*$, such that the following estimates hold:
		\begin{align}
			& \max_{t\in [\alpha^*, \gamma^*] \bigcup [\delta^*,\beta^*] } \mathcal{B}_{t^*} (t)  < -c^*, and  \label{ch-sign-max1b}\\[1mm]
			& \min_{t\in [\alpha_*, \gamma_*] \bigcup [\delta_*,\beta_*] } \mathcal{B}_{t_*} (t)  >  c_*.   \label{ch-sign-min1b}
		\end{align}
	\end{enumerate}
\end{lemma}
\begin{proof}
	See Lemma 5.10 in \cite{AGK18}.
\end{proof}
\begin{proof}[Proof of Theorem \ref{nec1}]
	The proof is a variant of the one of Theorem 3.7 in \cite{AC} (see also Theorem 5.9 in \cite{AGK18}).
	With the same notation of Lemma \ref{lem:bchngsign}, consider the intervals
	\begin{equation*}
		I^* \equiv [\alpha^*, \gamma^*] \cup [\delta^*,\beta^*] \ \ \textrm{ and } \ \
		I_* \equiv [\alpha_*, \gamma_*] \cup [\delta_*,\beta_*],
	\end{equation*}
	and choose $g^*, \psi^* \in \mathcal{G}^{s}(\T)$ such that
	\begin{align*}
		&  \operatorname{supp}(\psi^*) \subset [0,2\pi] \ \mbox{ and } \ \psi^*|_{[\alpha^*, \beta^*]}\equiv 1, \\[2mm]
		&  \operatorname{supp}(g^*) \subset [\alpha^*, \beta^*]  \ \mbox{ and } \ g^*|_{[\gamma^*,\delta^*]}\equiv 1.
	\end{align*}

	We assume that $\lambda_{j}>0$ and define $\{u_j\} \subset \mathcal{G}^{s}(\T)$ by
	\begin{equation*}
		u_j(t) = g^*(t) \exp\left[\lambda_j \psi^*(t) (\mathcal{B}_{t^*}(t) - i A_{t^*}(t))\right],
	\end{equation*}
	where $\mathcal{A}_{\eta}(t) = \int_{\eta}^{t}a(s)ds.$ Then, if $t \in \mbox{supp}(g^*)$ we get
	\begin{equation*}
		u_j(t) =
		g^*(t) \exp \left[ \lambda_j(\mathcal{B}_{t^*}(t) - i  A_{t^*}(t))\right],
	\end{equation*}
	and  $e^{\lambda_j \mathcal{B}_{t^*}(t)}\leqslant 1$, since  $\mathcal{B}_{t^*}(t) \leqslant 0$ on $I^*$. Since $|u_{j}(t^*)| = 1$, for any $j$, we have
	$u_{j}\to u \in  \mathscr{F}'\setminus \mathscr{F}.$

	Next, consider the sequence 
	\begin{equation*}
		f_j(t) = 	-i {g^{*}}'(t) \exp\left[\lambda_j\psi^*(t)(\mathcal{B}_{t^*}(t) - i  A_{t^*}(t))\right].
	\end{equation*}
	Note that $\supp(f_{j}) \subset I^*$, for any $j \in \N^\ast$. Hence
	\begin{align*}
		\left| \partial_t^{\gamma} f_{j}(t) \right| & \leq 
		\sum_{\ell=0}^{\gamma}\binom{\gamma}{\ell}\left|\partial_t^{\gamma-\ell+1}\big(g^*(t)\big)\right| \ \left| \partial_t^{\ell}\left(\exp \left[ \lambda_j(\mathcal{B}_{t^*}(t) - i  A_{t^*}(t))\right]\right) \right| \\ 
		& \leq C_1^{\gamma+1}\sum_{\ell=0}^{\gamma}\binom{\gamma}{\ell} ((\gamma - \ell +1)!)^{s} |\lambda_j|^{\ell}		\exp(\lambda_j \mathcal{B}_{t^*}(t)) \\
		& \leq C_2^{\gamma+1} (\gamma!)^{s} \, |\lambda_j|^{\gamma}	\exp(\lambda_j \mathcal{B}_{t^*}(t)) \\
		& \leq   C_2^{\gamma+1} (\gamma!)^{s} \, |\lambda_j|^{\gamma} \exp(-c^* \lambda_j) \\
        & \leq   C_3^{\gamma+1} (\gamma!)^{s} \, j^{\gamma \rho} \exp(-c^* j^\rho),
	\end{align*}	
and employing Lemma \ref{lem:expA} we obtain
$$
 C_3^{\gamma+1} (\gamma!)^{s} \, j^{\gamma \rho} \exp(-c^* j^\rho) \leq  C_M^{\gamma+1} (\gamma!)^{s+1} |\lambda_j|^{-M}
$$
	for all $M \in \mathbb N$. Therefore,
	$$
f = \sum_{j \in \N^\ast}	f_{j}(t)\varphi_j \in  \mathscr{F}
	$$
implying that $\mathcal L$ is not globally hypoelliptic, since $\mathcal Lu=f$.
\end{proof}
\begin{rem}
	We remark that Theorem \ref{nec1} can be extended to the following case: there exist an interval $[t_0,t_1] \subset [0,2\pi]$ and $\delta >0$ such that
	\begin{align*}
		b(t) & > 0, \ \forall t \in(t_0 -\delta, t_0), \\
		b(t) & = 0, \  \forall t \in [t_0,t_1], \\
		b(t) & < 0, \ \forall t \in(t_1, t_1 + \delta).
	\end{align*}
	
	Indeed, in this case we may consider cutoff functions $g_0$ and $g_1$ such that
	\begin{align*}
		&  \operatorname{supp}(g_0) \subset [t_0-\epsilon,t_0+\epsilon] \ \mbox{ and } \ g_0|_{[t_0-\epsilon/2,t_0+\epsilon/2]}\equiv 1, \\[2mm]
		&  \operatorname{supp}(g_1) \subset [t_1-\epsilon,t_1+\epsilon]  \ \mbox{ and } \ g_1|_{[t_1-\epsilon/2,t_1+\epsilon/2]}\equiv 1,
	\end{align*}
	for $\epsilon>0$ sufficiently small. Also, we set
	\begin{align*}
		B_0(t) & = \int_{t_0}^{t} b(s)ds, \ t \in \operatorname{supp}(g_0), \\
		B_1(t) & = \int_{t_1}^{t} b(s)ds, \ t \in \operatorname{supp}(g_1).
	\end{align*}
	Therefore, the sequence
	\begin{equation*}
		u_j(t) =
		g_0(t) \exp \left[ \lambda_j(B_{0}(t) - i  A(t))\right] +
		g_1(t) \exp \left[ \lambda_j(B_{1}(t) - i  A(t))\right],
	\end{equation*}
	where $A(t) = \int_{0}^{t}a(s)ds$, satisfies
	$u_{j}\to u \in  \mathscr{F}'\setminus \mathscr{F}$ and $\mathcal Lu \in \mathscr{F}$.
\end{rem}
\begin{rem}
	The case $b\equiv0$ is more delicate, and requires the study of further properties of the space $\Fc$, which is currently in progress.
	Like in other functional settings, for the analysis of this case we rely on the hypoellipticity properties of a constant coefficient operator
	of the form \eqref{eq:L_constant_coeff}, naturally associated with $\caL$. 
\end{rem}

\appendix

\section{Scattering calculus on asymptotically Euclidean manifolds}\label{app:scattering}

In this Appendix, we collect the basic facts and concepts about asymptotically Euclidean (or \textit{scattering}) manifolds 
and the scattering calculus. For further details, we refer the reader, e.g., to \cite{ME} and \cite{CDS19}.

\begin{defn}
Let $X$ be a smooth compact manifold with boundary $\partial X$. We say that the function $\rho_X \in \mathscr{C}^\infty(X)$ is a boundary defining function for the manifold $X$ if it satisfies:
\begin{enumerate}
    \item $\rho_X > 0$ on $X \setminus \partial X$;
    \item $\rho_X = 0$ precisely on $\partial X$;
    \item $d\rho_X \neq 0$ on $\partial X$.
\end{enumerate}
\end{defn}

\begin{defn}
An asymptotically Euclidean (or scattering) manifold is a 
smooth compact manifold $X$ with boundary, equipped with a Riemannian metric $g$ which, near the boundary, takes the form
\[
g = \frac{d\rho_X^2}{\rho_X^4} + \frac{g_\partial}{\rho_X^2},
\]
where $\rho_X$ is a boundary defining function and $g_\partial$ is a smooth symmetric 2-tensor that restricts to a Riemannian metric on $\partial X$.
\end{defn}

Such manifolds are also referred to as \emph{asymptotically Euclidean manifolds} or \emph{asymptotically flat manifolds}, due to the the fact that their Riemann curvature tensor is vanishing at the boundary.

One of the main points of interest in these manifolds lies in the fact that non-compact manifolds can be studied by considering the interior \( X \setminus \partial X \). They form a \emph{universal} class, as clarified in the following Proposition \ref{prop:universal}.

\begin{prop}
\label{prop:universal}
Any smooth compact manifold with boundary admits a scattering metric and can therefore be regarded as a scattering manifold.
\end{prop}

For the Euclidean space \( \mathbb{R}^d \), the most common compactifications are given by the \emph{radial compactification} and the \emph{stereographic projection}. These constructions proceed as follows.

\paragraph{Radial compactification.}

Consider the closed unit ball \( \mathbb{B}^d = \{x \in \mathbb{R}^d : |x| \leq 1\} \), with boundary \( \partial \mathbb{B}^d = \mathbb{S}^{d-1} \). Define a diffeomorphism \( \iota: \mathbb{R}^d \to \mathbb{B}^d{}^\circ \) by
\[
\iota(x) = \frac{x}{|x|} \left(1 - \frac{1}{|x|}\right), \quad |x| > 3,
\]
with inverse given by
\[
\iota^{-1}(y) = \frac{y}{|y|} (1 - |y|)^{-1}, \quad |y| \geq \frac{2}{3}.
\]
This \emph{radial compactification} introduces a differential structure at infinity. In polar coordinates, for large \( r \), the map becomes
\[
\iota(r, \varphi) = \left(1 - \frac{1}{r}, \varphi\right).
\]
A boundary defining function is given by
\[
\rho_{\mathbb{B}^d}(y) = \frac{1}{[\iota^{-1}(y)]}, \quad \text{where } [x] = |x| \text{ for } |x| > 3.
\]

\begin{prop}
In a collar neighborhood of \( \partial \mathbb{B}^d \), the Euclidean metric on \( \mathbb{R}^d \) pulls back to
\[
g = \frac{d\rho_{\mathbb{B}^d}^2}{\rho_{\mathbb{B}^d}^4} + \frac{g_{\mathbb{S}^{d-1}}}{\rho_{\mathbb{B}^d}^2},
\]
where \( g_{\mathbb{S}^{d-1}} \) denotes the standard round metric on the sphere.
\end{prop}

\paragraph{Stereographic Projection.}

Alternatively, the stereographic projection compactifies \( \mathbb{R}^d \) onto the upper hemisphere:
\[
\mathrm{SP}: \mathbb{R}^d \to \mathbb{S}^d_+ \subset \mathbb{R}^{d+1}, \quad z \mapsto \left( \frac{1}{\langle z \rangle}, \frac{z}{\langle z \rangle} \right), \quad \langle z \rangle = \sqrt{1 + |z|^2}.
\]
This transformation equips \( \mathbb{S}^d_+ \) with a scattering metric, where \( t_0 = \langle z \rangle^{-1} \) serves as a boundary defining function.

\begin{defn}
We define the space $\dot{\mathscr{C}}^\infty(\mathbb{S}_+^d)$ as the set of smooth functions vanishing to infinite order at the boundary:
\[
\dot{\mathscr{C}}^\infty(\mathbb{S}_+^d) = \bigcap_{k} \rho_{\mathbb{S}_+^d}^k \mathscr{C}^\infty(\mathbb{S}_+^d),
\]
where $\rho_{\mathbb{S}_+^d}$ is any boundary defining function.
\end{defn}

\begin{prop}
The stereographic pullback induces an isomorphism:
\[
\mathrm{SP}^*: \dot{\mathscr{C}}^\infty(\mathbb{S}_+^d) \rightarrow \mathscr{S}(\mathbb{R}^d).
\]
\end{prop}

We can now illustrate the scattering calculus of pseudo-differential operators on scattering manifolds. Following \cite{ME},
we give the definition in $\mathbb{S}_+^d$ (which is the compactification of $\R^d$) and then for a general scattering manifold $X$.

\begin{defn}
The class $\Psi_{\scc}^{m, \mu}(\mathbb{S}_+^d)$ consists of operators $A$ on $\dot{\mathscr{C}}^\infty(\mathbb{S}_+^d)$ such that the conjugated operator
\[
A'(\phi) = \mathrm{SP}^*(A(\mathrm{SP}_{*} \phi))
\]
belongs to $\Op(S^{m,\mu}(\mathbb{R}^d))$. The subscript ``scc'' stands for \emph{scattering conormal}.
\end{defn}

Applying stereographic compactification to both factors yields
\[
\mathrm{SP}_2: \mathbb{R}^d \times \mathbb{R}^d \to \mathbb{S}_+^d \times \mathbb{S}_+^d.
\]
Denote by $\rho_N$ and $\rho_\sigma$ the respective boundary defining functions in the two factors.

\begin{defn}
Consider the space $\mathcal{V}_{\sccl}(X) = \rho \mathcal{V}_b(X)$ of smooth vector fields which are the product of the boundary defining function $\rho$ and a smooth vector field on $X$ which is tangent to the boundary. We shall denote by $\mathrm{Diff}^*_{\sccl}(X)$ the enveloping algebra of $\mathcal{V}_{\sccl}(X)$, meaning the ring of operators on $\mathscr{C}^{\infty}(X)$ generated by $\mathcal{V}_{\sccl}(X)$ and $\mathscr{C}^{\infty}(X)$.
\end{defn}

\begin{rem}
In a collar neighbourhood of the boundary we have that $\mathcal{V}_{\sccl}(X)$ is spanned by $\rho^2 D_{\rho}$ and $\rho \mathcal{V}(\partial X)$, where $\mathcal{V}(\partial X)$ denotes the space of all smooth vector fields on $\partial X$.
\end{rem}

\begin{defn}
Let $\mathcal{V}_\mathrm{b}(\mathbb{S}_+^d \times \mathbb{S}_+^d)$ be the Lie algebra of vector fields tangent to all boundary hypersurfaces. The $L^\infty$-based conormal space of multi-order $(m, \mu)$ is defined as
\[
\mathcal{A}^{m,\mu} = \left\{ u \in \rho_N^{-\mu} \rho_\sigma^{-m} L^\infty : D u \in \rho_N^{-\mu} \rho_\sigma^{-m} L^\infty, \forall D \in \mathrm{Diff}^*_\mathrm{b} \right\}.
\]
\end{defn}

\begin{prop}
The pullback of conormal functions satisfies
\[
a \in S^{m,\mu}(\mathbb{R}^d \times \mathbb{R}^d) \iff a \in \mathrm{SP}_2^* \mathcal{A}^{m,\mu}.
\]
\end{prop}

\begin{defn}
The space $\Psi_{\sccl}^{m, \mu}(\mathbb{S}_+^d) \subset \Psi_{\scc}^{m, \mu}(\mathbb{S}_+^d)$ consists of operators whose Schwartz kernels lie in $\mathrm{SP}^*(\rho_N^{-\mu} \rho_\sigma^{-m} \mathscr{C}^\infty)$. These are called \emph{classical scattering pseudodifferential operators}.
\end{defn}

\begin{rem}
One has the identification
\[
\Psi_{\sccl}^{m,\mu}(\mathbb{S}_+^d) = (\mathrm{SP}^*)^{-1} \circ \Op(S_{\mathrm{cl}}^{m,\mu}) \circ \mathrm{SP}^*,
\]
where $S^{m,\mu}_{\mathrm{cl}}\subset S^{m,\mu}$ denotes the subclass of classical SG symbols (see, e.g., \cite{ABC25,BC11,CD21,MSS06,NR}).
\end{rem}

\begin{defn}
Let $\mathcal{V}_{\mathrm{b}}(X)$ be the Lie algebra of all smooth vector fields on $X$ which are tangent to the boundary. Then, we define the Lie algebra 
$$
\mathcal{V}_{\sccl}(X)=\rho \mathcal{V}_{\mathrm{b}}(X),
$$
whose structure bundle is denoted by
$$
{}^{\sccl} T X=1 / \rho \cdot{ }^{\mathrm{b}} \mathrm{T} X.
$$
${}^{\sccl} T ^{\ast}X$ will denote the dual of ${}^{\sccl} T X$.
\end{defn}

\begin{defn}
We define $^{\sccl} \bar{T}^{*} X$ as the compact manifold with corners of codimension two obtained by radial compactification of the fibres of ${ }^{\sccl} T^{*} X$. Denoting its boundary $\partial({ }^{\sccl} \bar{T}^{*} X)$ by $C_{\sccl}X$, we can see that it consists of two smooth manifolds with boundary, namely
$$
C_{\sccl}X={ }^{\sccl} S^{*} X \cup {}^{\sccl} \bar{T}_{\partial X}^{*} X,
$$
which meet at their boundaries, with a natural identification
$$
{ }^{\sccl} S^{*} X \cap {}^{\sccl} \bar{T}_{\partial X}^{*} X={ }^{\sccl} S_{\partial X}^{*} X,
$$
where ${ }^{\operatorname{sc}} S_{p}^{*} X$ is the sphere
$$
{ }^{\mathrm{s} c} S_{p}^{*} X=\left({ }^{\sccl} T_{p}^{*} X \backslash 0\right) / \mathbb{R}^{+},
$$
which represents the sphere at infinity of the radial compactification of the fibre. 
\end{defn}

In order to extend Definition \ref{def:psix} to a general scattering manifold, we proceed in the canonical way, that is, by cut-off and partition of unity. We will employ the notation ``sk'' meaning both ``sc'' and ``scc''. 

\begin{defn}\label{def:psix}
Let $X$ be a scattering manifold. Fix a covering of $X$ by coordinate systems $F: O \longleftrightarrow$ $O^{\prime} \subset \mathbb{S}_{+}^{d}$ where $O \subset X$ and $O^{\prime}$ are relatively open sets and $F$ is a diffeomorphism of manifolds with boundary. Moreover, let $m, \mu \in \mathbb{R}$. Then, the spaces $\Psi_{\mathrm{sk}}^{m, \mu}(X)$ consists of those linear operators $A: \dot{\mathscr{C}}^{\infty}(X) \longrightarrow \dot{\mathscr{C}}^{\infty}(X)$ such that:
\begin{enumerate}
\item if $\varphi, \psi \in \mathscr{C}^{\infty}(X)$ have disjoint supports, then $\varphi A \psi \in \Psi_{\mathrm{sk}}^{-\infty, -\infty}(X)$;
\item if $F: O \longrightarrow O^{\prime} \subset \mathbb{S}_{+}^{d}$ is a coordinates patch with $\varphi \in \mathscr{C}^{\infty}(X), \psi \in \mathscr{C}^{\infty}\left(\mathbb{S}_{+}^{d}\right)$ such that $\operatorname{supp}(\varphi) \subset O, \operatorname{supp}(\psi) \subset O^{\prime}$, then
\begin{equation}
A_{F, \psi, \varphi}=(F^{-1})^{*} \circ \psi \cdot A \circ F^{*} \circ \varphi \cdot \in \Psi_{\mathrm{sk}}^{m, \mu}(\mathbb{S}_{+}^{d}).
\end{equation}
\end{enumerate}
\end{defn}

The spaces $\Psi_{\mathrm{sk}}^{m, \mu}(X)$ have a number of properties: below, we state the most relevant.
\begin{prop}
The spaces $\Psi_{\mathrm{sk}}^{m, \mu}(X)$ form a bi-filtered algebra, that is
$$
\Psi_{\mathrm{sk}}^{m_{1}, \mu_{1}}(X) \circ \Psi_{\mathrm{sk}}^{m_{2}, \mu_{2}}(X) \subset \Psi_{\mathrm{sk}}^{m_{1}+m_{2}, \mu_{1}+\mu_{2}}(X), \quad m_1,m_2,\mu_1,\mu_2 \in \R.
$$
Furthermore, they contain the scattering differential operators:
$$
\operatorname{Diff}_{\sccl}^{m}(X) \subset \Psi_{\sccl}^{m, 0}(X), \quad m \in \mathbb{N}_{0}.
$$
There is also a basic conjugation-invariance under multiplication by powers of the boundary defining functions $\rho \in \mathscr{C}^{\infty}(X)$:
$$
\rho^{-M} \cdot \Psi_{\mathrm{sk}}^{m, \mu}(X)=\Psi_{\mathrm{sk}}^{m, \mu}(X) \cdot \rho^{-M}=\Psi_{\mathrm{sk}}^{m, \mu+M}(X).
$$
\end{prop}

\begin{defn}
A smooth map $\Psi: Y \to Z$ between compact manifolds with boundary is a \emph{scattering map (sc-map)} if for any $a \in \rho_Z^{-m} \mathscr{C}^\infty(Z)$, one has
\begin{enumerate}
    \item $\Psi^* a \in \rho_Y^{-m} \mathscr{C}^\infty(Y)$;
    \item positivity of $\rho_Z^m a$ at a point implies positivity of $\rho_Y^m \Psi^* a$ at preimages.
\end{enumerate}
\end{defn}

\begin{prop}
A local map $\Psi: U \subset Y \to V \subset Z$ is a local sc-map if and only if there exists $h \in \mathscr{C}^\infty(Y)$, $h > 0$, such that
\[
\Psi^* \rho_Z = \rho_Y h.
\]
\end{prop}
	
\bibliographystyle{plain} %

\end{document}